\documentclass[]{article}
\usepackage{amssymb,amsmath,amsthm}
\usepackage{enumitem}
\usepackage{xspace}
\usepackage{todonotes}
\usepackage[T1]{fontenc}
\usepackage[utf8]{inputenc}

\newcounter{dummy} \numberwithin{dummy}{section}
\newtheorem{theorem}[dummy]{Theorem}

\newtheorem{lemma}[dummy]{Lemma}
\newtheorem{corollary}[dummy]{Corollary}
\newtheorem{fact}[dummy]{Fact}
\let\oldfact\fact
\renewcommand{\fact}{\oldfact\normalfont}

\newtheorem{question}[dummy]{Question}
\let\oldquestion\question
\renewcommand{\question}{\oldquestion\normalfont}

\newtheorem{definition}[dummy]{Definition}
\let\olddefinition\definition
\renewcommand{\definition}{\olddefinition\normalfont}

\newtheorem{remark}[dummy]{Remark}
\let\oldremark\remark
\renewcommand{\remark}{\oldremark\normalfont}

\newtheorem{example}[dummy]{Example}
\let\oldexample\example
\renewcommand{\example}{\oldexample\normalfont}

\DeclareMathOperator{\diam}{diam}

\DeclareMathOperator{\compdiam}{comp-diam}

\DeclareMathOperator{\interior}{int}
\DeclareMathOperator{\RACG}{RACG}
\DeclareMathOperator{\Link}{Link}
\DeclareMathOperator{\Star}{Star}

\newcommand{\N}{\mathbb{N}}

\newcommand{\ZZ}{\mathbb{Z}} 

\newcommand{\Ttwo}{\mathbb{T}^2}
\newcommand{\RPtwo}{\mathbb{R}P^2}
\newcommand{\Stwo}{\mathbb{S}^2}
\newcommand{\Sone}{\mathbb{S}^1}

\newcommand{\sided}[1]{\overrightarrow{#1}}
\newcommand{\sidedop}[1]{\overleftarrow{#1}}
\newcommand{\cl}[1]{\overline{#1}}
\newcommand{\SL}[1]{\mathcal{#1}}
\newcommand{\Cells}[1]{\mathcal{#1}}
\newcommand{\embedsin}{\hookrightarrow}
\newcommand{\organism}{organism\xspace}
\newcommand{\mitosis}{mitosis\xspace}
\newcommand{\invlim}{\varprojlim}

\newcommand{\sidesof}{\leftthreetimes}
\newcommand{\actson}{\curvearrowright}
\newcommand{\nbd}[1]{\mathcal{U}_{#1}}
\newcommand{\homeo}{\cong}

\newcommand{\editA}[1]{{\color{black} {#1}}}

\title{Surface-like boundaries of hyperbolic groups}
\author{Benjamin Beeker and Nir Lazarovich\thanks{Supported by the Israel Science Foundation (grant No. 1562/19), and by the German-Israeli Foundation for Scientific Research and Development}}

\date{}

\begin{document}

\maketitle
\begin{abstract}
	We classify the boundaries of hyperbolic groups that have enough quasiconvex codimension-1 surface subgroups with trivial or cyclic intersections.
\end{abstract}

\section{Introduction}
By \cite{kahn2012immersing,bergeron2012boundary}, every fundamental group of a closed hyperbolic 3-manifold acts geometrically on a CAT(0) cube complex with surface subgroup hyperplane stabilizers. That is, the hyperplane stabilizers are virtually fundamental groups of closed surface subgroups of genus $\ge 2$. It is also known that the boundary of such a group is homeomorphic to $\Stwo$.

In \cite{markovic2013criterion} the converse is proved. That is, if $G$ is a hyperbolic group with $\partial G \cong \Stwo$ and $G$ acts geometrically on a CAT(0) cube complex with surface subgroup hyperplane stabilizers, then $G$ is virtually the fundamental group of a closed hyperbolic 3-manifold.
This gives a criterion for Cannon's Conjecture, which asks if $\partial G \cong \Stwo$ is sufficient to determine that $G$ is virtually the fundamental group of a closed hyperbolic 3-manifold.

In \cite{beeker2017detecting}, it is shown that assuming that $G$ is a one-ended hyperbolic group that acts geometrically on a CAT(0) cube complex with surface subgroup hyperplane stabilizers, and additionally has vanishing 1st cohomology with $\ZZ/2$ coefficients at infinity, then $\partial G \cong \Stwo$. The additional cohomological assumption is also necessary.

In this paper we continue our study of boundaries of one-ended hyperbolic groups that act geometrically on a CAT(0) cube complex with surface subgroup hyperplane stabilizers, and show the following classification under an additional assumption of `simple' intersections of hyperplanes.

\begin{theorem}\label{main thm}
	Let $G$ be a one-ended hyperbolic group that acts geometrically on a CAT(0) cube complex with surface subgroup hyperplane stabilizers. If all hyperplane stabilizers intersect in finite \editA{or} virtually $\ZZ$ subgroups then $\partial G$ is  homeomorphic to $\Stwo$, the Pontryagin sphere or the Pontryagin surface $\Pi(2)$.
\end{theorem}

Recall that the Pontryagin sphere can be obtained by taking the standard Sierpi\'{n}ski carpet and gluing together the opposite sides of each square as one glues a torus from a square. The surface $\Pi(2)$ is obtained in a similar way, by gluing the boundary of each square in the same way one glues a Klein bottle from a square.

\editA{
Let $G$ be a hyperbolic group. we say that $G$ has \emph{enough codimension-1 quasiconvex surface subgroups} if for every two points in $\partial G$ there exists a quasiconvex surface subgroup $H\le G$ such that $\partial H \subset \partial G$ separates the two points.
By results of Sageev \cite{sageev1995ends} and Bergeron-Wise \cite{bergeron2012boundary}, $G$ has enough codimension-1 quasiconvex surface subgroups if and only if $G$ acts geometrically on a CAT(0) cube complex with surface hyperplane stabilizers.
Hence, the assumption in Theorem \ref{main thm} can be changed to ``let $G$ be a one-ended hyperbolic group with enough codimension-1 quasiconvex surface subgroups with trivial or cyclic intersections''.}

\editA{
Examples of such groups include right-angled Coxeter groups $\RACG(L)$ where $L$ is the 1-skeleton of a triangulation of a closed surface without induced 4-cycles. $\RACG(L)$ acts on the associated Davis complex. The hyperplane stabilizers are isomorphic to $\RACG(\Star(v,L))$ for some vertex $v\in L$, where $\Star(v,L)$ is the star of the vertex of $v$ in $L$. Since the link $\Link(v,L)$ is a cycle graph, $\RACG(\Star(v,L))$ is virtually a surface group. The intersection of two hyperplane stabilizers is isomorphic to the $\RACG(\Star(v,L) \cap \Star(w,L))$ which is finite or virtually cyclic. However, the boundary of such groups is already known to be isomorphic to one of the spaces above by \cite{swiatkowski2013treesofmflds}.
}

The idea of the proof is to use the $\Sone$ boundaries of the hyperplanes to form a sequence of graphs that approximates $\partial G$. The way these graphs partition the boundary into components endows them with additional data that describes how to embed them in closed surfaces. The corresponding sequence of graphs in surfaces forms an inverse system whose inverse limit is homeomorphic to $\partial G$. Such inverse limits are known as \emph{trees of manifolds} (or as \emph{Jakobsche spaces}) and can be determined uniquely by the types of surfaces that appear in the sequence. These spaces were studied extensively, also in relation to boundaries of hyperbolic groups, in \cite{jakobsche1980bing,ancel1985construction,jakobsche1991homogeneous,stallings1995extension,zawislak2010trees,swiatkowski2013treesofmflds,swiatkowski2013treesofcompacta,swikatkowski2019reflection}. We will use a characterization due to Świątkowski \cite{swiatkowski2013treesofcompacta}. While the idea of the proof is straightforward, the proof is quite involved, as various topological anomalies need to be understood and addressed.

We end this introduction with the following question.

\begin{question}
	Assuming that $\dim \partial G=2$, is Theorem \ref{main thm} true without the additional assumption on intersections of hyperplanes?
\end{question}

\editA{Note that the boundary $\partial G$ must be of dimension $\le 2$ since it has enough separating circles.}
A full classification of 1-dimensional boundaries of groups that do not split over two-ended subgroups was given in \cite{kapovich2000hyperbolic}, and some work has been done in understanding the full picture of 1-dimensional boundaries (e.g \cite{swikatkowski2019reflection}). The classification of 2-dimensional boundaries is much less complete, and the question above can be seen as a step in this direction. 

Generically, quasiconvex surface subgroup tend to intersect along free groups and not on cyclic groups. Our theorem should be thought of as the `totally geodesic' surfaces case, and not the general quasiconvex case.

\editA{
\paragraph{Acknowledgements.} We would like to thank Michah Sageev for his comments, and  the anonymous referee for their valuable suggestions.
}

\section{A zoo of objects}

Let $G$ be a one-ended hyperbolic group, and assume that $G$ acts geometrically on a CAT(0) cube complex with surface hyperplane stabilizers that intersect in virtually cyclic groups. Let $\Xi$ be the boundary of the group $G$.

\subsection{The separation lemmas}

\begin{definition}
A \emph{hyperplane circle} or an \emph{H-circle} is the limit set in $\Xi$ of a codimension-1 quasiconvex surface subgroup in $G$.
A \emph{hyperplane arc} or an  \emph{H-arc}  is a subinterval of an H-circle.
\end{definition}


\begin{lemma}[Circles decompose into two components]\label{lem: two components}
	Let $\alpha$ be an H-circle, then $\Xi-\alpha$ consists of exactly two components. The boundary in $\Xi$ of each component is $\alpha$. Moreover, every H-circle $\beta$ that separates two points of $\alpha$ must intersect both components of $\alpha$.
\end{lemma}

\begin{proof}
	Let $A$ be a component of $\Xi-\alpha$. First we show that $\partial A = \alpha$. Otherwise $\partial A \subsetneq \alpha$ is a closed set. $\partial A$ is not a singleton since boundaries of one-ended hyperbolic groups have no cut points (by \cite{bowditch1998group,swarup1996cut}). Let $a$ and $a'$ be two distinct points of $\partial A$ that are the endpoints of a segment $\alpha'\subseteq \alpha - \partial A$. Let $\beta$ be an H-circle that separates $a,a'$. Then $\beta$ must intersect $A$ since otherwise there is a component $B$ of $\Xi-\beta$ such that $A \subseteq B$, but then $\{a,a'\}\subseteq \partial A \subseteq B$. By our assumption $|\beta\cap\alpha |=2$. Since $\beta$ separates $a,a'$ it must intersect $\alpha'$ in exactly one point $b$. Since $\beta$ intersects $A$ and $A$ is a connected component of $\Xi-\alpha$ we conclude that the point $b$ is in $\partial A$, contradicting the assumption on $\alpha'$.
	
	Let $\beta$ be an H-circle that separates two points $a,a'\in \alpha- \beta$. Then $\beta$ must intersect all the components of $\Xi-\alpha$. Otherwise, there is a component $A\subseteq \Xi-\alpha$ and a component $B\subseteq \Xi-\beta$ such that $A\subseteq B$. But then $a,a'\in B$.
\end{proof}

A connected component $A$ of the complement $\Xi - \alpha$ of an H-circle will be called an \emph{H-neighborhood}, and $\alpha$ will be called its \emph{bounding H-circle}.
It is easy to check that the H-neighborhoods form a basis for the topology of $\Xi$.

\begin{lemma}\label{lem: two state}
	Let $\alpha,\beta$ be two H-circles which intersect in two points $c,c'$. The two segments of $\alpha-\{c,c'\}$ are in the same component of $\Xi-\beta$ if and only if the two segments of $\beta-\{c,c'\}$ are in the same component of $\Xi-\alpha$.
\end{lemma}
\begin{proof}
	Let $\gamma$ be an H-circle separating $c,c'$, and let $a,a'$ and $b,b'$ be its intersections with $\beta$ and $\alpha$ respectively. Then, the cyclic order of intersection points on $\gamma$ is either $a,a',b,b'$ or $a,b,a',b'$ (up to interchanging $a,a'$ or $b,b'$). Since $\gamma$ must meet both sides of $\Xi-\beta$ and $\Xi-\alpha$, we can see that either the two segments of $\alpha-\{c,c'\}$ are separated by $\beta$ and the two segments of $\beta-\{c,c'\}$ are separated by $\alpha$ or none of them is separated by the other.
\end{proof}

This enables us to define transversality of intersecting H-circles and H-arcs.

\begin{definition}
	Let $\alpha$ and $\beta$ be intersecting H-circles. If $\alpha$ separates two points of $\beta$ (which by the previous lemma is equivalent to $\beta$ separating two points of $\alpha$), we say $\alpha$ and $\beta$ are \emph{transverse}. Similarly, if $\alpha',\beta'$ are intersecting H-arcs, we say that $\alpha'$ and $\beta'$ are transverse if the corresponding H-circles are transverse.
\end{definition}

The proof of Lemma \ref{lem: two state} also shows that for an H-circle $\gamma$ that separates the intersection points $c,c'$ of $\alpha,\beta$ the circular order in which it meets $\alpha,\beta$ is independent of $\gamma$, and depends only on whether $\alpha$ and $\beta$ are transverse.

\begin{lemma}\label{lem: 4 components}
	Let $\alpha,\beta$ be transverse H-circles bounding the H-neighborhoods $A,B$ respectively, let $c,c'$ be their two intersection points and let $\alpha_1,\alpha_2$ and $\beta_1,\beta_2$ be the two pairs of H-arcs bounded by $c,c'$ on $\alpha$ and $\beta$ respectively. Then $A\cap B,A \cap B^c, A^c\cap B, A^c\cap B^c$ are the 4 components of $\Xi - (\alpha\cup\beta)$  components, the boundary of each component is one of the 4 possible circles of the form $\alpha_i\cup\beta_j$.
\end{lemma}

\begin{proof}
	Let $C$ be a component of $\Xi-(\alpha\cup\beta)$. First we note that if $\cl{C}$ intersects the interior of the arc $\alpha_1$ then $\cl{C}$ contains $\alpha_1$. Otherwise, let $(a,a']$ be an interval in $\alpha_1 - \cl{C}\cap \alpha_1$ such that $a\in \cl{C}\cap \alpha_1$. Let $S$ be a small H-neighborhood separating $a$ and $a'$ and disjoint from $\beta$, and let $\sigma$ be its bounding H-circle. The H-circle $\sigma$ must intersect $C$, and thus the two intersection points of $\sigma$ with $\alpha_1$ must belong to $\cl{C}$, one of these two points must be in  $(a,a')$, giving a contradiction.
	Similarly if $\cl{C}$ intersects the interior of $\alpha_2$ (resp. $\beta_j$) then it contains the entire H-arc $\alpha_2$ (resp. $\beta_j$).
	
	It is also clear that there are at most 2 components of $\Xi-(\alpha\cup\beta)$ whose closure intersects the interior of some $\alpha_i$, because the bounding H-circle $\gamma$ of a small H-neighborhood around an interior point of $\alpha_i$ must intersect every such connect component, but $\gamma - (\alpha\cup\beta)$ has 2 components.
	
	If $\cl{C}$ intersects $c$ then it intersects the interior of some $\alpha_i$ and $\beta_j$. To show this, let $S$ be a small H-neighborhood of $c$ whose bounding H-circle $\sigma$ separates $c$ from some point of $C$, then as before $C$ must intersect $\sigma$, and thus must contain two intersection points of $\sigma$ with $\alpha\cup\beta$. Since $\alpha$ and $\beta$ are transverse, one of the intersection points is on $\alpha$, say on $\alpha_i$, and one is on $\beta$, say on $\beta_j$.
	
	Combining the above we see that any connected component of $\Xi - (\alpha\cup\beta)$ must intersect the interior of $\alpha_i$ and $\beta_j$ for some $i,j$. On the other hand, for each $k$ there are only 2 connected components whose closure intersect $\alpha_k$, and similarly for $\beta_k$. This shows that there are at most 4 components. Every connected component must be in one of 4 possible intersections of $A$ or $A^c$ with $B$ or $B^c$, and a small H-circle around $c$ shows that there is a component in each of the 4 intersections. 
	This shows these intersections are indeed the connected components of $\Xi - (\alpha\cup\beta)$.
%
%
\end{proof}

Similarly one proves the following.

\begin{lemma}\label{lem: 8 components}
	Let $\alpha,\beta,\gamma$ be pairwise transverse H-circles which bound the H-neighborhoods $A,B,C$ such that each of the H-circles separates the intersection points of the other two. Then the 8 possible intersections of $A,B,C$ and their complements are the 8 components of $\Xi - (\alpha\cup\beta\cup\gamma)$. \qed
\end{lemma}

\subsection{Piecewise hyperplane graphs}

Now we shall introduce the main tool -- PH graphs (piecewise-hyperplane graphs).

\begin{definition}
	Let $D,Y$ be topological spaces, a \emph{$D$ in $Y$} (resp. an \emph{embedded $D$ in $Y$}) is a continuous map (resp. a topological embedding) $D\to Y$ up to precomposition with a homeomorphism of $D$.
\end{definition}

\begin{definition}[PH-graphs]
	If $D$ is a finite (topological) graph whose vertices have degrees $\le 4$, we say that $D$ in $\Xi$ (resp. an embedded $D$ in $\Xi$) is a \emph{pre-PH-graph} (resp. \emph{embedded pre-PH-graph}) if $D$ has a graph structure so that every edge is sent homeomorphically to an H-arc. It is a \emph{PH-graph} (resp. \emph{embedded PH-graph}) if moreover the collection of edges incident to a vertex belong to a pair of transverse H-circles.
	
	Similarly, we define:
	\begin{description}
		\item[PH-path] a PH-graph with $D=[0,1]$.
		\item[PH-arc] an embedded PH-graph with $D=[0,1]$.
		\item[PH-circuit] a PH-graph with $D=\Sone$.
		\item[PH-curve] an embedded PH-graph with $D=\Sone$.
	\end{description}
\end{definition}

\editA{

The next lemma will help us upgrade any pre-PH-graph into a PH-graph by a small change that does not ruin the complementary components too much.

\begin{lemma}\label{lem: pre PH to PH}
	If $X$ is an embedded PH-graph, and $X\subseteq X'$ is an embedded pre-PH-graph which contains $X$, and $U_i$ are neighborhoods of the vertices of $X'$, then there is an embedded PH-graph $X''$ such that
	\begin{enumerate}
	    \item $X\subseteq X''$
	    \item $X''- \bigcup U_i = X' - \bigcup U_i$, and
	    \item the components of $\Xi-X''$ which are not contained in $\bigcup U_i$, are contained in components of $\Xi-X'$. 
	\end{enumerate}
\end{lemma}

\begin{proof}
	If $X'$ is an embedded pre-PH-graph, let $v_i\in U_i$ be the vertices of $X'$. If the H-arcs $\alpha_{i,1},\ldots,\alpha_{i,\deg(v)}$ of $X'$ incident to $v_i$ are not transverse, let $B_i$ be a small H-neighbourhood of $v_i$ which is contained in $U_i$ and does not intersect the H-arcs of $X'$ except for $\alpha_{i,1},\ldots,\alpha_{i,\deg(v)}$, and let $\beta_i$ be its bounding H-circle. Note that $\beta_i$ intersects $\alpha_{i,1},\ldots,\alpha_{i,\deg(v)}$ transversely.
	Replace $X'$ by the graph $X''$ defined by \[X'' = ( X' - \bigcup_i B_i )\cup \bigcup_i \beta_i \cup X\] where $i$ ranges over the index set of the vertices of $X'$. 

	Clearly the components of $\Xi - X''$ which are not contained in $\bigcup B_i$, must be contained in a component of $\Xi-X'$ as any path in $\Xi-X''$ which is not a path in $\Xi-X'$ must pass through some $B_i$, and hence must cross some $\beta_i$.
\end{proof}

The next lemma similarly shows that one can turn pre-PH-arcs into PH-arcs.

\begin{lemma}\label{lem: pre-PH-arc to PH-arc}
    If $X'$ is an embedded pre-PH-arc, and $U_i$ are neighborhoods of the interior vertices of $X'$, then there exists an embedded PH-arc $X''$ such that $X''- \bigcup U_i = X' - \bigcup U_i$. 
\end{lemma}
\begin{proof}
    Let $v_i$, $\alpha_{i,1},\alpha_{i,2}$ and $\beta_i$ be as in the previous proof. Set \[X'' = ( \gamma - \bigcup_i B_i )\cup \bigcup_i \beta'_i\] where $i$ ranges over the interior vertices of the arc $X'$, and $\beta'_i$ is an H-sub-arc of $\beta_i$ connecting the intersection points of $\beta\cap\alpha_{i,1}$ and $\beta\cap\alpha_{i,2}$.
\end{proof}
}

\begin{lemma} \label{lem: paths to PH-arcs}
	Let $p$ be an embedded arc in $\Xi$, there is a sequence of embedded PH-arcs that limits to $p$ in the compact-open topology (up to reparametrization). If moreover the endpoints of $p$ are on some H-arcs $\alpha,\alpha'$ then the endpoints of the PH-arcs can be chosen to be transverse intersections of the first and last H-arcs with $\alpha$ and $\alpha'$ respectively.
\end{lemma}

\begin{proof}
	Cover $p$ with small H-neighborhoods $B_i$ with bounding H-circles $\beta_i$, and without loss of generality for all $i\ne j$, $B_i$ does not contain any $B_j$. The set $\bigcup_i\beta_i$ is connected, and contains pre-PH-arcs $\gamma$ which are close to $p$ in the compact-open topology. Moreover, the endpoints of these can be chosen to be transverse to the H-arcs $\alpha,\alpha'$ by changing the first and last sets of the cover $B_i$. Using \editA{Lemma \ref{lem: pre-PH-arc to PH-arc}} we find a close embedded PH-arc $\gamma'$ which has the same properties. 
\end{proof}

\subsection{The space of sides}
%

\begin{definition}
	Let $Z\subset Y$ be a pair of topological spaces. 
	For a point $y\in Y$ let $\nbd{y}$ be the collection of open sets containing $y$.
	A \emph{side} of $y$ is a choice $C_U$ for every $U\in \nbd{y}$ of a connected component of $U- Z$ such that $C_U \subset C_V$ if $U\subset V$.
	In other words, a side of $y$ is a germ of connected components of $y$ in $Y- Z$.
	Note that if $y\in Z^\circ$ then it has no side.
	A \emph{sided point} is a pair  $\sided{y}:=(y,\{C^{\sided{y}}_U\}_{U\in \nbd{y}})$ of a point $y$ and a side. We will denote a sided point by $\sided{y}$.
	Let $Y\sidesof Z$ be the set of sided points.
	We endow $Y\sidesof Z$ with the topology given by the local basis $V^{\sided{y}} _U$ is the set of all $\sided{y}'$ such that $y'\in U$ and $C^{\sided{y}}_U = C^{\sided{y'}}_U$.
	\editA{We denote by $F:Y\sidesof Z \to Y$ the \emph{forgetful map} mapping $\sided{y} \mapsto y$.}
\end{definition}

Here is an example that one should keep in mind.

\begin{example}\label{ex: surface with graph}
	Let $\Omega$ be a closed surface, and let $O$ be a graph embedded in $\Omega$. Then $\Omega \sidesof O$ is the compact surface with boundary obtained by cutting $\Omega$ along $O$.
\end{example}

\begin{remark}
	Note the following:
	\begin{enumerate}
		\item 	The forgetful map $F:Y\sidesof Z \to Y$ is continuous. 
		In Example \ref{ex: surface with graph}, the map $F$ is the map that glues $\Omega$ back from $\Omega \sidesof O$.
		
		\item If $Y$ is locally connected and $Z$ is closed then any connected component of $Y- Z$ embeds topologically in $Y\sidesof Z$, and connected components of $Y\sidesof Z$ correspond to connected components of $Y-  Z$. For a connected component $C$ in $Y- Z$ we will denote by $\tilde{C}$ its corresponding component in $Y\sidesof Z$. Explicitly, $\tilde{C}$ is the set of $\sided{y}$ such that $C^{\sided{y}}_U \subset C$ for all $U\in\nbd{y}$.
		We denote by $\partial \tilde{C}$ the set $\tilde{C} - C$ (where $C$ is viewed as a subspace of $Y\sidesof Z$).
		We denote by $\partial C$ the map $F$ restricted to $\partial \tilde{C}$.
		In Example \ref{ex: surface with graph}, $\tilde{C}$ are the 'regions' of the complement of the graph on the surface, and $\partial C$ are their parametrized boundary curves (note that the boundary curves are not necessarily embedded curves).

		\item Let $\alpha$ be an H-arc (or an H-circle) in $\Xi$, then Lemma \ref{lem: two components} implies that $F:\Xi- \alpha \to \Xi$ is 2-to-1 on $\alpha^{\circ}$ (where $\alpha^{\circ}$ denotes the interior of $\alpha$ as an arc) and 1-to-1 on $\Xi - \alpha^{\circ}$. More precisely, there are two arcs (or curves) $\sided{\alpha},\sidedop{\alpha}$ in $\Xi \sidesof \alpha$ whose interiors map homeomorphically to  the interior of $\alpha$ under $F$.
	\end{enumerate}
\end{remark}

The following is an easy corollary of Lemma \ref{lem: 4 components}
 
\begin{lemma}\label{lem: 2 sides}
	Let $\alpha$ be an H-circle in $\Xi$, and let $a\in \alpha$, then $a$ has two sides, $\sided{a},\sidedop{a}$. In fact, $F^{-1} \alpha \simeq \sided{\alpha} \sqcup \sidedop{\alpha}$ and each of $\sided{\alpha},\sidedop{\alpha}$ is an embedded circle of $\Xi\sidesof \alpha$.
\end{lemma}

\begin{proof}
	It is enough to define the side using a local basis of $a$. Let $\nbd{a}^H$ be the local basis of H-neighborhoods of $a$, then by lemma \ref{lem: 4 components} if $B\in \nbd{a}^H$ then $B-\alpha$ has two components, namely $B\cap A$ and $B\cap A^c$ where $A,A^c$ are the H-neighborhoods of $\alpha$. The choice of $U^{\sided{a}}_B = B\cap A$ and $U^{\sidedop{a}}_B = B\cap A^c$ are thus the only two sides of $a$.
	
	Similarly it is easy to see that $F^{-1} \alpha \simeq \sided{\alpha} \cup \sidedop{\alpha}$, where $\sided{\alpha} = \{\sided{a}|a\in\alpha\}$ and $\sidedop{\alpha} = \{\sidedop{a} | a\in \alpha\}$ and $\sided{a},\sidedop{a}$ are defined above.
\end{proof}

Similarly one proves the following using Lemma \ref{lem: 8 components}

\begin{lemma}\label{lem: 4 sides or less}
	Let $\alpha,\beta$ be transverse H-circles in $\Xi$ and let $c,c'\in \alpha\cap \beta$ be their 2 intersection points, and let $\alpha_i,\beta_i$, $i=1,2$, be the two subsegments of $\alpha,\beta$ oriented from $c$ to $c'$. Then
	\begin{enumerate}
		\item The point $c$ (and $c'$) has 4 sides.
		\item The sides of $\alpha\cup\beta$ form 4 circles as follows
		\[F^{-1}(\alpha\cup\beta) \simeq \sided{\alpha}_1\sided{\beta}_1^{-1} \sqcup \sidedop{\beta}_1\sided{\alpha}_2^{-1} \sqcup \sidedop{\alpha}_2\sidedop{\beta}_2^{-1} \sqcup \sided{\beta}_2\sidedop{\alpha}_1 ^{-1}\]
		where $\sided{\alpha_i},\sidedop{\alpha_i},\sided{\beta_j},\sidedop{\beta_j}$ are defined as above using some H-neighborhoods $A,B$ bounded by $\alpha,\beta$ respectively. Moreover, these cycles of the boundaries $\partial C$ of the connected components $C\subseteq \Xi - (\alpha\cup\beta)$ viewed in $\Xi\sidesof(\alpha\cup\beta)$.
		\item \label{lem: degree of vertex sides} Similarly if we replace $\alpha\cup\beta$ by some union $Z$ of $k$ of the 4 H-arcs $\alpha_i,\beta_i$, $i=1,2$, then $c$ has $k$ sides and $F^{-1}(Z)\subset \Xi - Z$ is the obvious disjoint union of circles.
	\end{enumerate}\qed
\end{lemma}

It is useful to note that if $c$ is a point on a PH-graph, then a neighborhood of $c$ fits into \ref{lem: degree of vertex sides} of Lemma \ref{lem: 4 sides or less}.
Let $c$ be on a PH-graph $X$, and let $C$ be a small H-neighborhood of $c$ that meets $X$ only at $c$ and the edges adjacent to $c$, then the $\deg(c)$ components of $C- X$ will be called the \emph{sided H-neighborhoods of $c$}.

On the set of edges incident to a vertex, the adjacency of edges in boundary cycles, defines a natural unoriented cyclic order. We will refer to unoriented cyclic order as \emph{a circular order}.
We summarize Lemma \ref{lem: 4 sides or less} as follows. 

\begin{lemma}\label{lem: sides of PH-graphs}
	If $X$ is a PH-graph, for each connected component $C$ of $\Xi- X$,  $\partial C\subseteq \Xi\sidesof X$ is a union of circles. The map $F:\tilde{X} = F^{-1}(X)\to X$ defines a local circular order on the edges around each vertex. \qed
\end{lemma}

We will call connected components of $\Xi - X$ the cells of $(X,\Xi)$. In the next subsection, we will introduce the right framework to describe this structure.


\subsection{Organisms and animals}
%

\begin{definition}
	A \emph{fat graph} is a graph $X$ and a choice of circuits in $X$, \editA{called \emph{boundary circuits},} $\partial X:\coprod\Sone\to X$, such that the gluing of a disk along each circuit gives a closed surface.
	Equivalently, each edge has two preimages under this map, and around every vertex the adjacency of edge in the circuits defines a circular order.
	Each lift of an edge to $\partial X$ is called a sided edge.
	
	A fat graph is \emph{orientable} if the closed surface is orientable, or equivalently if one can orient the circuits such that they now define a cyclic order around vertices.
	
	Let $X_1, X_2$ be fat graphs, we say that $X_1\subseteq X_2$ \emph{as fat graphs} if $X_1\subseteq X_2$ as graphs and the circular ordering at the vertices of $X_1$ is the same as that induced from $X_2$.
\end{definition}
    \editA{We note that if $X_1\subseteq X_2$ as fat graphs then each sided edge of $X_1$ corresponds to a unique sided edge of $X_2$.}
\begin{definition}
	An \emph{\organism} is a pair  $\SL{X} = (X,\Cells{C})$ where $X$ is a connected fat graph, and $\Cells{C}$ is a partition of the circuits of $\partial X$. The elements of $\Cells{C}$ are called \emph{cells}, the set of boundary circuits of a cell $C\in \Cells{C}$ is denoted by $\partial C$.
	
	For organisms $\SL{X}_1,\SL{X}_2$ we say that $\SL{X}_1\subseteq \SL{X}_2$ \emph{as organisms} if $X_1\subseteq X_2$ as fat graphs and two sided edges of $\SL{X}_1$ belong to the same cell if they are equivalent with respect to the equivalence relation$\sim$ on sided edges of $X_2$ generated by $e \sim e'$ if $e,e'$ belong to the same cell in $\SL{X}_2$ or they are the two sides of an edge of $X_2 - X_1$. Two sided edges of $X_1$ are in the same cell of $\SL{X}_1$ if the are $\sim$-equivalent.
	
\end{definition}

\begin{definition}
	Let $\Delta$ be a topological space, and let $D$ be an embedded graph in $\Delta$. The pair $(D,\Delta)$ is an \emph{animal} if: 
	\begin{enumerate}
		\item $\Delta$ is connected, locally connected, compact and metrizable.
		\item The boundaries $\partial \tilde C$ of the components in $\Delta\sidesof D$ are embedded curves, and the maps $\partial C$ endow $\Delta$ with an organism structure. 
	\end{enumerate}
	We will abuse terminology and refer to the components of $\Delta \sidesof D$, the components of $\Delta - D$, and the boundaries $\partial \tilde{C}$ as the cells of the animal.
\end{definition}

\begin{example}
	\begin{itemize}
		\item \emph{Surface animals}. The main example of an animal is an embedded graph in a closed surface. 
		Observe that any organism can be seen as induced from some surface animal by gluing connected surfaces with boundaries along the boundary circuits of each cell.
		\item \emph{Tame animals}. If in a surface animal the components $\tilde{C}$ are holed spheres we will call such an animal a \emph{tame animal}.
		\item  \emph{Wild animals}. The second example which will be important for us is an embedded PH-graph in $\Xi$, as Lemma \ref{lem: sides of PH-graphs} shows. We will call $(X,\Xi)$ a \emph{wild animal}.
	\end{itemize}
\end{example}

\begin{lemma}
	If $X_1 \subseteq X_2$ are embedded PH-graphs in $\Xi$, then $\SL{X}_1\subseteq \SL{X}_2$ as organisms.
\end{lemma}

\begin{proof}
	The fat graph structure of $\SL{X}_1$ embeds in $\SL{X}_2$ because the structure is determined locally by small H-circles around the vertices.
	The cell structure has to do with connected components of $\Xi - X_1$ and so, \editA{every cell of $X_2$ is in a cell of $X_1$.}
\end{proof}

\begin{definition} 
A PH-curve $A$ is \emph{2-sided} if $\tilde{A}\subset \Xi\sidesof A$ has two components. Otherwise the PH-curve is \emph{1-sided}.
We say that $\Xi$ is \emph{orientable} if it does not admit 1-sided PH-curves.
\end{definition}

For example, an H-circle is 2-sided.

An isomorphism of fat graphs is a homeomorphism of the graphs which respects the boundary.

\subsection{Embryogenesis}

\begin{quotation}
	``Embryology will often reveal to us the structure, in some degree obscured, of the prototype of each great class.''
	--- Charles Darwin
\end{quotation}

\begin{definition}
	An \emph{elementary extension} of a fat graph $X$ is a fat graph $X'$ which as a graph is obtained from $X$ by adding an edge $e$, and $X\subset X'$ as fat graphs. 
	Equivalently, all the boundary circuits of $X'$ except those who pass through $e$ are boundary circuits of $X$.
	
	The circuits of $X$ that are not circuits of $X'$ are the \emph{divided circuits} of the extension. \editA{The circuits that are in $X'$ but not in $X$ are the \emph{new circuits} of the extension.}
\end{definition}

%
%

Every fat graph can be obtained from a point by a sequence of elementary extensions.


\begin{definition} 
	Let $X'$ be an elementary extension of $X$, if $e$ has only one endpoint in $X$, then we say that the extension is of \emph{type 0}.
	Otherwise, the extension is of \emph{type 1} (resp. \emph{2}) if it has $1$ (resp. $2$) divided circuits. 
	We can further divide into cases as follows (see Figure \ref{fig:elem_ext}):
	\begin{itemize}
		\item If $X'$ is a type 0 extension of $X$, then $X'$ is uniquely determined by the edge $e$ and the divided circuit.
		
		\item There are 2 possible type 1 extensions of $X$ that share the same underlining graphs and divided circuits. These extensions can be distinguished by counting the number of circuits of $X'$ which are not circuits of $X$. This number can be $1$ or $2$, we call them type 1.1 and type 1.2.
		
		\item There are 2 possible type 2 extensions of $X$ that share the same underlining graphs and divided circuits. However, they always produce only one new circuit. We call them \emph{twin} extensions. We will see next that with an additional structure of a tame animal, one of the twins will be preferred.  
		
	\end{itemize}
\end{definition}

\begin{figure}[h]
	\centering{
		\resizebox{\linewidth}{!}{
			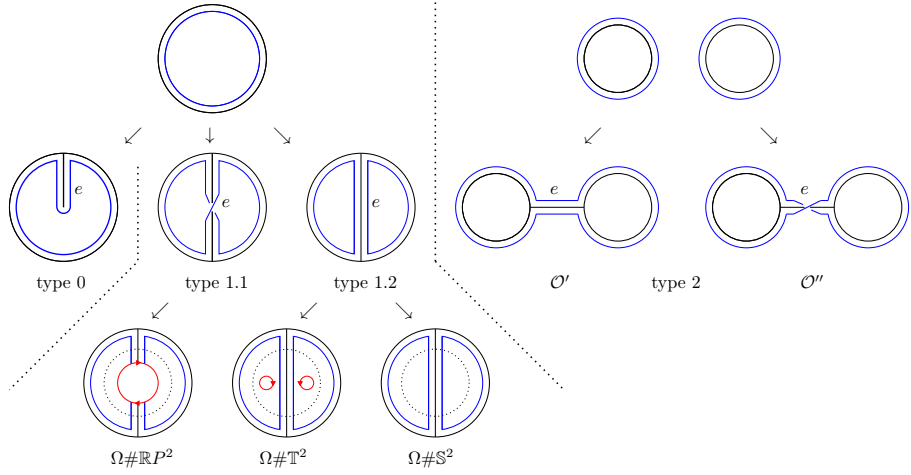
		}
		\caption{Elementary extensions. First row shows a part of $X$ (in black), and its changed boundary circuits (in blue). The second row shows the different elementary extensions after adding the added edge $e$, the new circuits are in blue. The last row shows the possible connected sum correction to the tame animal to accommodate a type 1 extension. The dotted line is the curve along which the connected sum is made, the red should be glued according to the orientation indicated by the arrow.}
		\label{fig:elem_ext}
	}
\end{figure}


%

\begin{definition}
	A \emph{\mitosis} in an \organism $\SL{X}$, is an organism $\SL{X}' \supseteq \SL{X}$ such that the underlining fat graph $X'$ of $\SL{X}'$ is an elementary extension of the underlining fat graph $X$ of $\SL{X}$\editA{, 
	such that the divided circuits of the elementary extension belong to the same cell $C$ of $\SL{X}$, all the other cells of $\SL{X}$ are cells of $\SL{X}'$, and the cell $C$ is replaced by at most two cells of $\SL{X}'$ each containing at least one of the new circuits.}
\end{definition}

\begin{remark}
	Note that only a mitosis corresponding to an elementary extension of type 1.2 might increase the number of cells. If this happens, the two new boundary circuits belong to different cells and the number of cells increases by 1.
\end{remark}

\begin{lemma}\label{lem: type 1 mitosis}
	Let $\SL{O}$ be the organism structure of a tame animal $(O,\Omega)$, and let $\SL{O}'$ be a mitosis of $\SL{O}$ of type 1. Then, $\SL{O}'$ is the organism structure of a tame animal $(O',\Omega')$ such that $\Omega'$ is a connected sum of $\Omega$ with either $\Ttwo$  or $\Stwo$ or $\RPtwo$ along a disk $D$ in a cell of $\Omega$ and $O\embedsin O'$ comes from $O\embedsin \Omega - D \embedsin \Omega'$.
\end{lemma}

\begin{proof}
	See Figure \ref{fig:elem_ext}. Note that in all cases, the new regions remain holed spheres.
\end{proof}

\begin{remark}
	In the previous lemma, if the mitosis is of type 1.1 then $\Omega' = \Omega \# \RPtwo$ and if the mitosis is of type $1.2$ then $\Omega' = \Omega \# \Stwo$ or $\Omega' = \Omega \# \Ttwo$ depending on if $\SL{O}'$ has more cells or the same number of cells as $\SL{O}$.
\end{remark}

\begin{lemma}
	In the setting of the previous lemma. If $\Omega = \Stwo$ and $\Omega'  = \Omega \# \Ttwo$, then a boundary circuit of a cell in $\SL{O}'$ is non-separating in $\Omega'$. If $\Omega$ is orientable and $\Omega' = \Omega \# \RPtwo$ then there is a M\"{o}bius band in $O'$ (i.e, a 1-sided circuit). \qed
\end{lemma}

\begin{lemma}\label{lem: type 2 mitosis}
	Let $\SL{O}$ be the organism structure of a tame animal $(O,\Omega)$, and let $\SL{O}',\SL{O}''$ be 2 twin mitoses of $\SL{O}$ of type 2. Then, exactly one of $\SL{O}',\SL{O}''$, say $\SL{O}'$, is the organism structure of a tame animal $(O',\Omega')$ such that $\Omega' = \Omega$ and  $O\embedsin O'$ comes from $O\embedsin \Omega = \Omega'$.
\end{lemma}

\begin{proof}
	See Figure \ref{fig:elem_ext}.
\end{proof}

\begin{lemma}\label{lem: cancer implies unoriented}
	In the setting of the previous lemma. The other twin, $\SL{O}''$, is non-orientable (i.e, there is a M\"obius band in $\SL{O}''$).\qed
\end{lemma}

\begin{definition}
	A mitosis of a tame animal is \emph{benign} if it is either of type 0 or 1, or the twin $\SL{O}'$ as in Lemma \ref{lem: type 2 mitosis}. Otherwise, we call it \emph{cancerous}.
\end{definition}

\editA{
We summarize Lemmas \ref{lem: type 1 mitosis}, \ref{lem: type 2 mitosis}, \ref{lem: cancer implies unoriented} as follows.
\begin{lemma}\label{lem: summary of extensions}
    Let $(O,\Omega)$ be a tame animal, and let $\SL{O}$ be its organism structure. Let $\SL{O}'$ be a benign mitosis of $\SL{O}$. Then there exists a tame animal $(O',\Omega')$ with organism structure $\SL{O}'$, which is obtained from $(O,\Omega)$ by a connected sum with $\Stwo, \Ttwo, \RPtwo$ along a disk $D$ in a cell of $\Omega$, such that $O\embedsin O'$ comes from $O\embedsin \Omega -D \embedsin \Omega'$.
\end{lemma}

\begin{remark}\label{rem: choosing disks properly in mitoses}
We note that one can choose the disk $D$ so that it avoids a finite collection of curves on $\Omega$.
\end{remark}
}
The next lemma will show that cancerous mitoses occurring in wild animals can be avoided. More precisely, it shows that every cancerous mitosis can be further extended by an elementary extension, such that the new organism can also be obtained by performing two benign mitoses instead. 

\begin{lemma}[Cancer treatment]\label{lem: cancer treatment}
	Let $(X,\Xi)$ be a wild animal, and let $(O,\Omega)$ be a tame animal whose organism structure $\SL{O}$ is isomorphic to the organism structure $\SL{X}$ of $(X,\Xi)$. Let $e$ be a PH-arc in $\Xi$ such that $X'=X\cup\{e\}$ is an embedded PH-graph in $\Xi$. If the organism structure $\SL{X}'$ of the wild animal $(X',\Xi)$, is a cancerous mitosis of $\SL{O}$, then we can find another PH-arc $e'$ in a neighborhood of a point on $e$  such that $X'' = X\cup\{e,e'\}$  is isomorphic to a tame animal $(O'',\Omega'')$ that can be obtained from $(O,\Omega)$ by two benign mitoses.
\end{lemma}

\begin{proof}
	By Lemma \ref{lem: cancer implies unoriented}, $X'$ is non-orientable, and hence $\Xi$ contains an embedded 1-sided PH-curve $M$. 
	By the uniform convergence dynamics of $G\actson \Xi$ we can find a 1-sided PH-curve $M$ in a small H-neighborhood $U$ around a point of $e$ and disjoint from $X$. By Lemma \ref{lem: paths to PH-arcs}, connect two points $p', p''$ on $M$ to some points $q',q''$ on $e$  by disjoint PH-arcs $\sigma',\sigma''$ respectively, both in the neighborhood $U$, and disjoint from $X'$ except for its endpoints $q',q''$ at which it is transverse to $e$.
	
	Let $q',q''$ break $e$ into 3 PH-arc segments $f',f,f''$. And let $p',p''$ break $M$ into two PH-arcs $M',M''$ oriented from $p'$ to $p''$.
	
	Consider the following two PH-arcs 
	\begin{align*}
	\rho' &= f'\cdot\sigma'^{-1}\cdot M' \cdot\sigma''\cdot f'',	\\
	\rho'' &= f'\cdot\sigma'^{-1}\cdot M''\cdot\sigma''\cdot f''.
	\end{align*}
	Note that since $M$ is 1-sided the PH-arcs $\rho',\rho''$ define twin mitoses of $X$. One of them is benign, say $\rho'$.
	Note that $X'' = X\cup \{\rho,f\} = X\cup \{e,e'\}$ where $e' = \sigma'^{-1}\cdot M'\cdot\sigma''$, and that $X''$ is obtained by a benign mitosis of type 2 along $\rho$ and then a (benign) mitosis of type 1 along $f$.
\end{proof}

\begin{figure}[h]
	\centering{
		\resizebox{0.6\linewidth}{!}{
			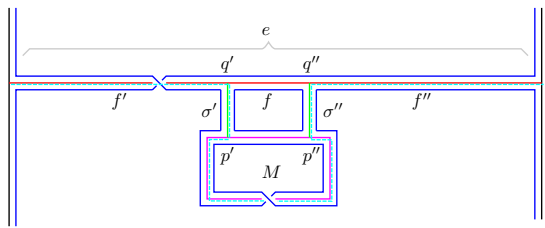
		}
		\caption{Cancer Treatment Lemma. The path $\rho'$ is marked as a dotted light blue line.}
		\label{fig:elem_ext}
	}
\end{figure}

\begin{remark}\label{rmk: safe cancer treatment}
	Note that the arcs added in Lemma \ref{lem: cancer treatment} can be chosen to have trivial intersection with any finite collection of PH-arcs.
\end{remark}

\section{Proof of the main theorem}

\subsection{The limiting process and strategy of proof}

\begin{definition}
	Let $Z$ be a compact metric space. Let $Z'\subseteq Z$. We denote by $\compdiam(Z')$ the supremum of diameters of components of $Z'$.
	Note that if $Z_n \subseteq Z$ is a sequence of subsets, then $\compdiam(Z_n)\to 0$ is a topological property that does not depend on the metric. It can be defined by saying that for every open cover $\mathcal{U}$ of $Z$, there exists $n$ such that each of the components of $Z_n$ \editA{is} contained in a set of the open cover $\mathcal{U}$.
\end{definition}

In what follows, the spaces we consider are always compact metrizable, and so we will use $\compdiam(Z_n)\to 0$ even in the absence of a fixed metric.

\begin{lemma}[Limit argument]\label{lem: limit argument}
	Let $\Xi$ and $\Omega$ be compact metrizable spaces. Let $X_n$ be an increasing sequence of embedded compact graphs in $\Xi$ such that $(X_n,\Xi)$ are animals, and let $f:\bigcup_n X_n \to\Omega$ be an injective \editA{function such that $f|_{X_n}$ is continuous for all $n$.} Assume that $(f(X_n),\Omega)$ are animals, and $f|_{X_n}$ induces an organism isomorphism between the corresponding organism structures. If $\compdiam(\Xi - X_n) \to 0$ and $\compdiam(\Omega - f(X_n))\to 0$, then $f$ extends uniquely to a homeomorphism $\Xi \to \Omega$.
\end{lemma}

\begin{proof}
	Endow $\Xi$ and $\Omega$ with some compatible metric. It is clear that $\Xi=\overline {\bigcup_n X_n}$. We first claim that it suffices to show that $f$ is uniformly continuous. \editA{To see this, assume that we have shown that $f$ is uniformly continuous. Since $\Omega$ is a compact metric space, it is a complete metric space. A uniformly continuous function with a range which is a complete metric spaces, can be extended to the closure of its domain. Hence, $f$ extends to a continuous function, which we will call $f$ as well,  $f:\Xi=\overline {\bigcup_n X_n}\to \Omega$.}  Since $\compdiam(\Omega-f(X_n))\to 0$ it is clear that $\overline{f(\bigcup_n X_n)}= \Omega$ and since $\Xi$ is compact we get that $f$ is surjective.  Since \editA{for any two points in $\Xi$ there exists $n$ such that $X_n$ separates them}, and components of $\Xi-X_n$ are mapped bijectively to components of $\Omega - f(X_n)$ it is clear that the extension is injective.
	
	To show that $f$ is uniformly continuous, let $\epsilon>0$ and let $n$ be such that $\compdiam(\Omega-f(X_n))<\epsilon/3$.
	By uniform continuity of $f$ on $X_n$, let $\delta'$ be such that for all $x'_1,x'_2\in X_n$, $d(x'_1,x'_2)<\delta'\implies d(f(x'_1),f(x'_2))<\epsilon/3$. 
	Now let $\delta$ be such that if $x_1,x_2\in \bigcup_n X_n$ are at distance $\delta$ then there are two points $x'_1,x'_2$ in the boundaries of the components of $x_1,x_2$ in $\Xi-X_n$ respectively at distance at most $\delta'$. 
	
	Let $x_1,x_2\in\bigcup_n X_n$ be such that $d(x_1,x_2)<\delta$, and let $d(x'_1,x'_2)<\delta'$ as in the definition of $\delta$. By the choice of $n$, $d(f(x_i),f(x'_i))\le\epsilon/3$, and by the choice of $\delta'$, $d(f(x'_1),f(x'_2))<\epsilon/3$. So, $d(f(x_1),f(x_2))<\epsilon$ as required.
\end{proof}

The strategy of the proof of Theorem \ref{main thm}, is to find an increasing sequence of wild animals $(X_n,\Xi)$, and in parallel to find isomorphic increasing sequence of tame animals $(O_n,\Omega_n)$ by a `back-and-forth' argument, at each time guaranteeing that $\compdiam$ of both $(X_n,\Xi)$ and of $(O_n,\Omega_n)$ limits to 0. We thus divide the argument between the following two subsections. In Subsection \ref{grow a wild animal}, given $X_n$ in $\Xi$ and $\epsilon>0$ we show how one finds $X_{n+1}$ with $\compdiam(\Xi- X_{n+1})<\epsilon$ and an isomorphic $(O_{n+1},\Omega_{n+1})$. In Subsection \ref{tame a wild animal}, given $(O_n,\Omega_n)$ and an open cover we show how to find $(X_{n+1},\Xi)$ and an isomorphic $(O_{n+1},\Omega_{n+1})$ extending $(O_n,\Omega_n)$ such that $\compdiam(\Omega_{n+1}-O_{n+1})$ refines the (induced) open cover. In Subsection \ref{Jakobsche} we show that the inverse limit of $(O_n,\Omega_n)$ is one of the three possibilities described in Theorem \ref{main thm} using a theorem of Jakobsche. 
\editA{Finally, in Subsection \ref{subsec: proof of main thm}, we put these ingredients together and prove Theorem \ref{main thm}.}

\subsection{How to grow a wild animal}\label{grow a wild animal}


\begin{lemma}\label{lem: refine wild partition}
	Given a PH graph $X$ in $\Xi$ and $\epsilon>0$, there is a PH graph $X'$ in $\Xi$ that extends $X$ and has $\compdiam(X')<\epsilon$.
	\editA{Moreover, if there exists a non-planar (resp. unorientable) PH graph in $\Xi$, then $X'$ can be assumed to contain a non-planar (resp. unorientable) PH-sub-graph in component of $\Xi-X$.}
\end{lemma}

\begin{proof}
	\editA{Let $Y$ is a non-planar (resp. unorientable) PH graph in $\Xi$, if such exists. By uniform convergence of $G\actson \partial G = \Xi$, we can find translates $g_1,\ldots,g_n \in G$ such that every component of $\Xi-X$ contains one of the translates $g_1 Y,\ldots,g_n Y$ of $Y$.
	
	Let $Z$ be a collection of H-circles such that $\compdiam(Z)<\epsilon$, which includes all H-circles which are used to define $g_1 Y \cup \ldots \cup g_n Y$. Add the circles in $Z$ to $X$, to form an embedded pre-PH-graph $\tilde{X}$. }
	Let $U_i$ be neighborhoods of the vertices of $\tilde{X}$ of diameter $<\epsilon$. Apply Lemma \ref{lem: pre PH to PH} to obtain a PH-graph $X'$ with $\compdiam(X')<\epsilon$, as any component of $\Xi-X'$ is either contained in a component of $\Xi-\tilde{X}$ (and hence in a component of $\Xi-X$) or in some $U_i$. 
	
	\editA{It follows from the construction in Lemma \ref{lem: pre PH to PH} that in the neighborhood of each $g_i Y$ the graph $X'$ will contain a non-planar (resp. unorientable) PH-graph.}
\end{proof}


\begin{lemma}\label{lem: extending animals}
	\editA{Let $(X,\Xi)$ be a wild animal, and $(O,\Omega)$ a tame animal with isomorphic organism structures, $\SL{X},\SL{O}$ respectively.} Let $X'$ be a PH-graph extension of $X$ with then there is a further extension $X''$ of $X'$ such that $X''$ is obtained from $X$ by a sequence of benign elementary extensions, that is, there is a tame animal $(O'',\Omega'')$ with the same organism structure \editA{$\SL{O}''$} as that of $(X'',\Xi)$, and $\SL{O}''$ is obtained by a sequence of benign extensions of  $\SL{O}$. \qed
\end{lemma}

\editA{
\begin{proof}
We can obtain the PH-graph $X'$ from $X$ by adding an arc at a time, thus getting a sequence $X=X_0\subset \ldots \subseteq X_n=X'$ of elementary extensions, whose corresponding organism structures $\SL{X}_i$ undergo mitoses. 
Let $e_i$ be the PH-arc that is added to $X_i$ to form $X_{i+1}$.
Inductively, we build a sequence $X=Y_0\subset \ldots Y_n=:X''$ of PH-graphs and a sequence $(O,\Omega)=(O_0,\Omega_0),\ldots,(O_n,\Omega_n)$ of tame animals, such that $X_i\subseteq Y_i$, $(Y_i - X_i)\cap X'=\emptyset$, and the organism structure of $(O_i,\Omega_i)$ is isomorphic to that of $(Y_i,\Xi)$.

At the $i$-th step of the induction, the organism structure of $(O_i,\Omega_i)$ is isomorphic to that of $(Y_i,\Xi)$. Note that since $(Y_i - X_i)\cap X'=\emptyset$ the PH-arc $e_i$ is disjoint from $Y_i$. Let $Y_{i+1}'=Y_i \cup \{e_i\}$ be the elementary extension of $Y_i$ obtained by adding $e_i$. 

If the elementary extension is benign, set $Y_{i+1} = Y_i'$, and by Lemma \ref{lem: summary of extensions} perform an isomorphic elementary extensions on $(O_i,\Omega_i)$ to form the tame animal $(O_{i+1},\Omega_{i+1})$ with isomorphic organism structure as $(Y_{i+1},\Xi)$. 

If the elementary extension if cancerous, perform cancer treatment in the following sense. By Lemma \ref{lem: cancer treatment}, there exists a PH-arc $e_i'$ such that the organism structure of the PH-graph $Y_{i+1} = Y_i \cup \{e_i,e_i'\}$ is isomorphic to a tame animal $(O_{i+1},\Omega_{i+1})$ that can be obtained from $(O_i,\Omega_{i+1})$ by two benign mitoses. By Remark \ref{rmk: safe cancer treatment} we may assume $e_i'$ is disjoint from $X'$, thus maintaining the assumption $(Y_{i+1} - X_{i+1})\cap X'=\emptyset$.
\end{proof}

}

\editA{As we will see, t}his completes one half of the proof. In the next subsection we show that any reasonable subdivision of the tame animal can be mimicked in the wild animal.

\subsection{How to tame a wild animal} \label{tame a wild animal}
%

\begin{lemma}\label{PONG}
	Let $(X,\Xi)$ be a wild animal, and let $(O,\Omega)$ be a tame animal with isomorphic organism structures, and let $f:\SL{X}\to\SL{O}$ be an isomorphism of the organisms of $(X,\Xi)$ and $(O,\Omega)$ respectively.
	For every open cover $\mathcal{U}$ of $\Omega$ there exists an PH-graph extension \editA{$X'$ of $X$}, an extension $(O',\Omega')$ of $(O,\Omega)$ with a map $F:\Omega'\to\Omega$ obtained by connected sums in the cells of $\Omega-O$, and an isomorphism $f':\SL{X}'\to \SL{O}'$ of the corresponding organisms extending $f$, \editA{i.e $F \circ f'|_X = f$}, and such that each component of $\Omega'-O'$ is in a subset of $F^{-1}(U)$ for some $U\in\mathcal{U}$.
\end{lemma}

\begin{proof}
	
The summarize the two steps we will perform, details of each step will follow.

Step 1. Extend $O$ by adding arcs so that each cell has one boundary component. Find similar PH-arcs in $\Xi$ extending $X$.

Let us denote by $(O'',\Omega'')$ and $(X'',\Xi)$ these isomorphic extensions (note that $\Omega'' = \Omega$).

Step 2. Each cell of $(O'',\Omega'')$ is a disk. Add arcs to $O''$ to divide each disk into components which are subsets of $\mathcal{U}$. Trying to mimic this in $\Xi$ using open sets coming from hyperplanes we see that we create some additional boundary circles. To deal with them we extend $\Omega=\Omega''$ to $\Omega'$ by introducing connected sums with small tori. This process is detailed below.

\subsubsection*{Step 1: Adding an arc of type 2}

Let $(X,\Xi)$ and $(O,\Omega)$ as in Lemma \ref{PONG}, and let $(O^\dagger, \Omega)$ be a type 2 extension, i.e, adding an arc $e^\dagger$ on a cell of $(O,\Omega)$ connecting two distinct boundary cycles. Let $\sided{o},\sided{o}'\in O$ be the two (sided) endpoints of the added arc $e^\dagger$.
Let $\sided{a},\sided{a}'\in X$ be their preimages under $f$. We may assume that $\sided{a},\sided{a}'$ are in the interior of the sided H-arcs $\sided{\alpha},\sided{\alpha}'$ of $X$ (as slightly perturbing the endpoints of the edge in will not affect the property that the components will be contained in $\mathcal{U}$). 

Note that $\sided{a}$ and $\sided{a'}$ are in the same cell of $(X,\Xi)$, and so there is an arc $\gamma$ in this cell, connecting $\sided{a}$ and $\sided{a'}$.
Cover $\gamma$ by small H-neighborhoods $B_1,\ldots,B_n$, bounded by the H-circles $\beta_1,\ldots,\beta_n$, in $C$ (here $\beta_1,\beta_n$ are understood to be half-H-circles which define the sided neighborhoods of $\sided{a},\sided{a}'$). On the union $\beta_1\cup\ldots\cup\beta_n$ there is a path which is piecewise H-arc that connects two points $c, c'$ on $\beta,\beta'$, and so together with some segments in $\beta,\beta'$ we get a pre-PH-arc path from $a$ to $a'$. \editA{By Lemma \ref{lem: pre-PH-arc to PH-arc}} we can replace it by a PH-arc $\gamma'$.

Following the same lines as the proof of Lemma \ref{lem: cancer treatment}, if $\Xi$ is non-orientable, then one can find some $\gamma''$ such that $\gamma'$ and $\gamma''$ are the twin extensions. Let us denote by $\gamma'$ the one that defines the same organism structure as $e^\dagger$ defines on $\Omega$. Note that if $\Xi$ is oriented then so is $\Omega$, and $\gamma'$ constructed above already has the same organism structure as $e^\dagger$ as only one of the twin defines extends orientably.

In this construction, the endpoints of $\gamma$ can be chosen arbitrarily close to $a$ and $a'$.  Let $e$ be the arc in $\Omega$ which is identical to $e^\dagger$ except for a neighborhood of their endpoints, and the endpoints of $e$ are the images of the endpoints of the PH-arc $\gamma'$ under the map $f$. This way $f$ extends to a map from $X\cup\gamma'$ to $O\cup e$.

\subsubsection*{Step 2: Adding arcs of type 1}

Performing Step 1 enough times we may assume that each cell of $(O,\Omega)$ has one boundary cycle. At this point, one needs to perform extensions of type 1 to reduce the size of cells (to fit into sets of $\mathcal{U}$). Let $e_1,\ldots,e_n$ be the edges added to a cell $C$ of $(O,\Omega)$, in this order, so that the components of $C-(e_1\cup\ldots\cup e_n)$ are in subsets of $\mathcal{U}$.

We will describe how to add $e_1$, the edges $e_2,\ldots,e_n$ are added inductively in the same way.

Let us denote by $e=e_1$. Let $\partial C$ be the boundary circuit of $C$ and let $\sided{o},\sided{o}'$ be the endpoints of $e \in \partial C$. Let $D$ and $\partial D$ be the corresponding cell and boundary in $(X,\Xi)$, and let $\sided{a}, \sided{a'}$ be the preimages of $\sided{o},\sided{o}'$ under $f$. Let $\delta\subseteq \partial D$ be one of the two arcs connecting $\sided{a},\sided{a'}$. Let $A_1,\ldots,A_m$ be a cover of $\delta$ by half-H-neighborhoods in $\Xi\sidesof X$, such that $A_1$ and $A_m$ contain $\sided{a},\sided{a'}$ respectively and are as small as needed, and each $A_i\cap \partial D$, for $1<i<m$, is in $\delta$.  Let $A= \bigcup_j A_j$. Note that $\partial A$ contains a PH-arc $\gamma$ between small neighborhoods of $\sided{a},\sided{a'}$, it might contain more components, which are closed PH-curves, $\gamma_1,\ldots,\gamma_s$ which are contained in $D$. 

Note that $\partial A$ divides $D$ into at least two components -- the component $A$ and a component of $D-A$ that intersects $\partial D - \delta$. Any other components of $D-A$ is contained in the interior of $D$, and thus, by adding more H-neighborhoods we may cover it completely, eliminating it without adding more components. Therefore we may assume that $\partial A$ divides $D$ into 2 components. Similar to the construction of PH-arcs in the previous step, for each $1\le i\le s$, let $\beta_i$ (resp. $\beta_i'$) be PH-arcs connecting  $\gamma_i$ to $\gamma$ in $A$ (resp. in $D-A$) whose endpoints on $\gamma$ are ordered as $\beta_1\cap\gamma,\beta_1'\cap\gamma,\ldots,\beta_s\cap\gamma,\beta_s'\cap\gamma$. Moreover, we require that if $\Xi$ is non-orientable then the sub-organism made by $\gamma\cup\gamma_i\cup\beta_i\cup\beta_i'$ is oriented by replacing $\beta_i,\beta_i'$ if needed as in Lemma \ref{lem: cancer treatment}.

Let $U_1,\ldots,U_s$ be small disjoint disc neighborhoods of $s$ points on $e$ in $\Omega$ (each intersecting $e$ in an arc and not intersecting any of $e_2,\ldots,e_n$). Let $\Omega'$ be the connected sum of $\Omega$ with $s$ tori. I.e. $\Omega'$ is the gluing of $s$ holed tori $T_1,\ldots,T_s$  and $\Omega - (U_1\cup \ldots \cup U_s)$ along the corresponding boundaries. Let $F:\Omega'\to \Omega$ be a continuous map such that $F$ is a homeomorphism between $\Omega' - (T_1\cup\ldots\cup T_s)$ and $\Omega - (U_1\cup\ldots\cup U_s)$.

As in Step 1, we change $e$ slightly so that its endpoints coincide with the images of the endpoints of $\gamma$ under the map $f$.
Let $e'$ be the arc in $\Omega'$ which is the union of $e - (U_1\cup\ldots\cup U_s)$ and $s$ arcs on $T_1,\ldots,T_s$. For each $1\le i\le s$, let $u_i$ be the simple closed curve on $T_i$ which is disjoint from $e'$. Note that $E=e'\cup(u_1\cup \ldots\cup u_s)$ divide $F^{-1}(C)$ into two components, $B_1$ and $B_2$. Without loss of generality $B_1$ is the components that contains $f(\delta)$. Let $t_i$ (resp. $t_i'$) be arcs in $T_i\cap B_1$ (resp. $T_i\cap B_2$) connecting $u_i$ to $e$, and such that the endpoints of $t_i$ and $t_i'$ on $e$ are ordered as $t_1\cap e,t_1'\cap e,\ldots,t_s\cap e,t_s'\cap e$. See Figure \ref{fig:PONG}.

Now extend $f$ to $X\cup (\gamma\cup\bigcup_{i=1}^s(\gamma_i\cup\beta_i \cup \beta_i')) \to O\cup (e'\cup\bigcup_{i=1}^s(u_i\cup t_i \cup t_i'))$ in the obvious way.
\end{proof}

\begin{figure}[h]
	\centering{
		\resizebox{\linewidth}{!}{
			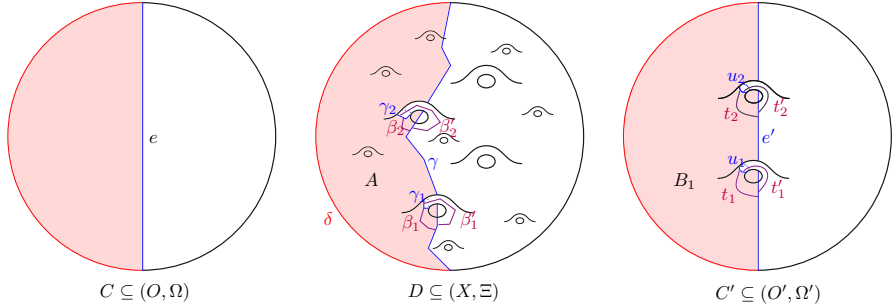
		}
		\caption{Step 2 of taming the wild animal.}
		\label{fig:PONG}
	}
\end{figure}

\subsection{What have we created?!}\label{Jakobsche}
\editA{The goal of this subsection is to describe the final ingredient of the proof that will enable us to classify the objects we have created as the inverse limit $\invlim \Omega_n$}. To do that, we use the construction by Jakobsche \cite{jakobsche1991homogeneous} of trees of manifolds. We give an account of this construction by Świątkowski \cite{swiatkowski2013treesofcompacta} adapted to our \editA{2-dimensional} setting.

\begin{definition}[Świątkowski \cite{swiatkowski2013treesofcompacta}]
	\editA{Let $\mathcal{M}$ be one of the following collection of closed surfaces, $\mathcal{M} = \{\Stwo\}, \{\Stwo,\Ttwo\}$ or $\{\Stwo,\Ttwo,\RPtwo\}$.
	A \emph{2-dimensional Jakobsche weak system} associated to the collection $\mathcal{M}$
    is an inverse sequence $\mathcal{J} = (\{\Omega_n\}_{n\in\N}, \{\phi_n\}_{n\in\N_{> 1}})$ of closed surfaces $\Omega_n$, with a compatible metric $d_n$, and continuous maps $\phi_n:\Omega_{n} \to \Omega_{n-1}$, 
    together with a closed embedded disk $D_n \subseteq \Omega_n$ such that:}
	\editA{
	\begin{enumerate}[label=(J\arabic*)]
		\item \label{J: S2}$\Omega_1 = \Stwo$;
		\item \label{J: conn sum}$\Omega_{n+1}$ is obtained from $\Omega_n$ by a connected sum with $M\in\mathcal{M}$ along the disk $D_n$;
		\item \label{J: conn sum 1}the map $\phi_{n}$ restricts to a homeomorphism from $\phi_{n}^{-1} (\Omega_{n-1} - \interior(D_{n-1}))$ to $\Omega_{n-1} - \interior(D_{n-1})$;
		\item \label{J: D_n are disjoint} for all $m>n\in\N$ we have $\phi_{m,n}(D_m)\cap \partial D_n = \emptyset$ where $\phi_{m,n} = \phi_{m}\circ\ldots\circ \phi_{n+1}$;
		\item \label{J: denseness} for all $n\in \N$ the union \[\bigcup \{\phi_{m,n}(D_m)\;|\;m\ge n\}\] is dense in $\Omega_n$; furthermore, in the case of $\mathcal{M} = \{\Stwo,\Ttwo\}$ we want the union corresponding to disks on which we perform a connected sum with $\Ttwo$ to be dense, and in the case $\mathcal{M} = \{\Stwo,\RPtwo,\Ttwo\}$ we want the union corresponding to disks on which we perform a connected sum with $\RPtwo$ to be dense.
		\item \label{J: diam D_n to 0} $\diam_{d_n} (\phi_{m,n}(D_m))\to  0$ as $m\to \infty$ where the diameter is with respect to the metric $d_n$ on $\Omega_n$.
	\end{enumerate}
	The inverse limit $\invlim \Omega_n$ is the \emph{Jakobsche space} of the 2-dimensional Jakobsche weak system $\mathcal{J}$, associated to the collection $\mathcal{M}$.}
\end{definition}

\editA{
\begin{remark}
\begin{enumerate}
    \item In a general Jakobsche weak system, as defined in \cite{swiatkowski2013treesofcompacta}, the collection of manifolds $\mathcal{M}$ is an arbitrary collection of compact manifolds (of the same dimension). Therefore, extra caution is used in defining the connected sums. However in dimension two the above conditions suffice. 
    \item Another small difference is that in \cite{swiatkowski2013treesofcompacta} multiple connected sums are allowed at each step.
    \item A more significant difference between this account of the definition and the original one is that in \ref{J: denseness} one should require that for every $M\in \mathcal{M}$ the collection \[\bigcup \{\phi_{m,n}(D_m)\;|\;m\ge n, \Omega_{m+1} = \Omega_m \# M\}\] is dense in $\Omega_n$.
    In our setting, a connected sum with $\Stwo$ is trivial so can be ommited. If $\mathcal{M}= \{\Stwo,\Ttwo,\RPtwo\}$, a priori we do not require that this collection is dense for $\Ttwo \in \mathcal{M}$. However, our assumption is sufficient since $\RPtwo \# \RPtwo \# \RPtwo = \RPtwo \# \Ttwo$.
\end{enumerate}

\end{remark}

\begin{theorem}[{Świątkowski,  \cite[Theorem~1.3]{swiatkowski2013treesofmflds}}]\label{thm: uniqueness of Jakobsche}
Given a Jakobsche weak system, the inverse limit $\invlim \Omega_n$ is uniquely determined by the collection $\mathcal{M}$.
\end{theorem}

In particular, we get the following characterisation of Jakobsche spaces for surfaces.
\begin{corollary}\label{cor: limits of Jakobsche}
Given a Jakobsche weak system as above,
\begin{enumerate}
    \item if $\mathcal{M} = \{\Stwo\}$ then $\invlim \Omega_n = \Stwo$,
    \item if $\mathcal{M} = \{\Stwo,\Ttwo\}$ then $\invlim \Omega_n$ is the Pontryagin sphere, and
    \item if $\mathcal{M} = \{\Stwo,\Ttwo,\RPtwo\}$ then $\invlim \Omega_n$ is the Pontryagin surface $\Pi(2)$.
\end{enumerate}
\end{corollary}
\begin{proof}
Each of the spaces above -- the sphere, the Pontryagin sphere, and the Pontryagin surface $\Pi(2)$ -- can easily be described as an inverse limit of a Jakobsche weak system with the corresponding collection $\mathcal{M}$. The corollary now follows from the uniqueness in Theorem \ref{thm: uniqueness of Jakobsche}.  
\end{proof}
}

\editA{
\subsection{Proof of Theorem \ref{main thm}}\label{subsec: proof of main thm}
	As in the rest of the paper, let $G$ be a one-ended hyperbolic group that acts geometrically on a CAT(0) cube complex with surface subgroup hyperplane stabilizers which intersect in finite or virtually $\ZZ$ subgroups.
	Let $\Xi = \partial G$.
	
	For $n\in \N$ we will build 
	\begin{itemize}
	\item a wild animal $(X_n,\Xi)$, with organism structure $\SL{X}_n$,
	\item a tame animal $(O_n,\Omega_n)$, with a metric $d_n$ on $\Omega_n$, with organism structure $\SL{O}_n$,
	\item if $n>1$, a continuous map $\phi_n : \Omega_{n} \to \Omega_{n-1}$, and
	\item organism isomorphisms $f_n:\SL{X}_n \to \SL{O}_n$
	\end{itemize}
	such that 
	\begin{enumerate}[label = (P\arabic*)]
	    \item the system $\mathcal{J}=(\{\Omega_n\}_{n\in \N},\{\phi_n\}_{n\in\N})$ is a 2-dimensional Jakobsche weak system for the collection $\mathcal{M}=\{\Stwo\},\{\Stwo,\Ttwo\}$ or $\{\Stwo,\Ttwo,\RPtwo\}$;
	    \item \label{P: increasing wild animals} $X_n \subseteq X_{n+1}$;
	    \item \label{P: compatability of f_n} the disk $D_n$ is disjoint from $O_n$ and the maps $f_n$ are compatible with the inverse system, i.e $\phi_{n+1}\circ f_{n+1}|_{X_n} = f_n$;
	    \item \label{P: compdiam in Xi}$\compdiam(X_{n_k},\Xi)\to 0$;
	    \item \label{P: compdiam in Omega} for every choice of components $C_m \subseteq \Omega_m - O_m$, $m\ge n$, $\diam _{d_n}\phi_{m,n}(C_m) \to 0$ as $m\to \infty$ where $\phi_{m,n} = \phi_m \circ \ldots \circ \phi_{n+1}$.
	\end{enumerate}
	
	Let us first show that this data suffices for the proof of the theorem. 
	Let $\Omega = \invlim \Omega_n$, and let $\pi_n:\Omega \to \Omega_n$ be the associated projections.
	By the universal property of inverse limits, and property \ref{P: compatability of f_n}, the sequence $f_n$ defines an injective map $f:\bigcup _n X_n \to \Omega$ such that $f|_{X_n}$ is continuous for all $n$. 
	
	Since $\pi_n ^{-1} (O_n) = f(X_n)$, the map $\pi_n:\Omega \to \Omega_n$ defines a continuous map $\pi_n^\sidesof:\Omega \sidesof f(X_n) \to \Omega_n \sidesof O_n$. 
	The components of $\Omega \sidesof f(X_n)$ are the preimages under $\pi_n^\sidesof$ of the components of $\Omega_n \sidesof O_n$. 
	In particular, $(f(X_n),\Omega)$ are animals and $f$ defines an isomorphism of organisms between the organism structure of $(X_n,\Xi)$ and that of $(f(X_n),\Omega)$.
	Property \ref{P: compdiam in Xi} states that $\compdiam(X_n,\Xi) \to 0$, and by property \ref{P: compdiam in Omega}, we have $\compdiam(f(X_m),\Omega) \to 0$. 
	By Lemma \ref{lem: limit argument}, $\Xi \homeo \Omega$, and by Corollary \ref{cor: limits of Jakobsche} we get that $\Xi$ is homeomorphic to either a sphere, a Pontryagin sphere or the Pontryagin space $\Pi(2)$. This is what we wanted to prove.}
	
    \bigskip
	
	\editA{
	We will construct the above data in the following way.
	We will construct a sequence $n_k\in \N$, such that $1=n_0<n_1<n_2<\ldots$, and for each $k\in \N$ we will construct $X_m,O_m,\Omega_m,d_m,\phi_m,f_m,D_{m-1}$ for $n_{k-1}<m\le n_k$.
	We do so by induction on $k$.
	In order to make sure that conditions \ref{J: denseness},\ref{J: diam D_n to 0},\ref{P: compdiam in Xi},\ref{P: compdiam in Omega} are satisfied, we will replace them by the following quantitative versions:
	\begin{enumerate}[label = (J\arabic*')]  \setcounter{enumi}{4}
	    \item \label{J': denseness}For every component $C_{n_{k-1}}\subseteq \Omega_{n_{k-1}} - O_{n_{k-1}}$ there exists $n_{k-1}\le m < n_{k}$ such that $\phi_{m,n_{k-1}}(D_m) \subseteq C_{n_{k-1}}$; moreover, if $\mathcal{M} = \{\Stwo,\Ttwo\}$ then $D_m$ can be assumed to be a disk on which a connected sum with $\Ttwo$ is performed, and if $\mathcal{M} = \{\Stwo,\Ttwo,\RPtwo\}$ then $D_m$ can be assumed to be a disk on which a connected sum with $\RPtwo$ is performed.
	\end{enumerate}
	\begin{enumerate}[label = (P\arabic*')]  \setcounter{enumi}{3}
        \item \label{P': compdiam in Xi} $\compdiam(X_{n_k},\Xi)\le \frac{1}{k}$.
	    \item \label{P': compdiam in Omega} For every choice of components $C_{n_k} \subseteq \Omega_{n_k}- O_{n_k}$, and for all $s\le n_{k-1}$, \[\diam_{d_{s}}( \phi_{n_k,s}(C_{n_k})) \le \frac{1}{k}.\] 
	\end{enumerate}
	
	Note that \ref{J: denseness} follows from \ref{J': denseness} and \ref{P': compdiam in Omega}; \ref{J: diam D_n to 0} follows from \ref{P': compdiam in Omega} and \ref{P: compatability of f_n}; \ref{P: compdiam in Xi} and \ref{P: compdiam in Omega} follow from \ref{P': compdiam in Xi} and \ref{P': compdiam in Omega} respectively.

    \bigskip
    
	We begin our induction with $k=0$. 
	Set $n_0=1$. 
    Let $X_1$ be an H-circle in $\Xi$. 
    Set $\Omega_1=\Stwo$ and let $O_1$ be a simple closed curve on $\Omega_1$. 
    By Lemma \ref{lem: 2 sides} and by the Jordan Curve Theorem, both $X_1 \subseteq \Xi$ and $O_1\subseteq \Omega_1$ are two sided curves, thus there exists an isomorphism $f_1:\SL{X}_1 \to \SL{O}_1$.
    By the choice of $\Omega_1$, \ref{J: S2} is satisfied.
    
    Set $\mathcal{M}$ according to the following cases.
    \begin{description}
        \item[Planar. ] If all PH-graphs are planar, set $\mathcal{M}=\{\Stwo\}$.
        \item[Orientable non-planar. ] If all PH-graphs are orientable, and there is a non-planar PH-graph in $\Xi$, set $\mathcal{M}=\{ \Stwo,\Ttwo\}$.
        \item[Non-orientable.  ] If there exists a non-orientable PH-graph in $\Xi$, set $\mathcal{M}=\{\Stwo,\Ttwo,\RPtwo\}$.
    \end{description}
    
    Now, for $k>0$. Assume that $X_m,O_m,\Omega_m,f_m,D_{m-1},\phi_m$ are defined for all $m\le n_k$ and satisfy the properties above. Let us construct it for all $n_k< m \le n_{k+1}$. We will do so in two steps.
    
    \paragraph{Step 1. Refining the wild animal.}
    Using Lemma \ref{lem: refine wild partition} we can find a PH-graph extension $X'$ of $X_{n_k}$ such that $\compdiam(X',\Xi) \le \frac{1}{k+1}$ and every component of $\Xi-X_{n_k}$ contains a non-planar (resp. a non-orientable) PH-graph if such exists. 
    By induction, the organism structure of $(X_{n_k},\Xi)$ is isomorphic to that of $(O_{n_k},\Omega_{n_k})$ by the map $f_{n_k}$. By Lemma \ref{lem: extending animals} we can find a further extension $X''$ of $X'$, and a tame animal $(O'',\Omega'')$ which are isomorphic, and such that $(O'',\Omega'')$ is obtained from $(O,\Omega)$ by a sequence of $s$ benign mitoses.
    By Lemma \ref{lem: summary of extensions} these benign mitoses can be realized by a sequence of connected sums.
    Set $m_k = n_k + s$, and let $(O_m,\Omega_m)$, for $n_k \le m \le m_k$, be the sequence of benign mitoses, starting with $(O_{n_k},\Omega_{n_k})$ and ending with $(O_{m_k},\Omega_{m_k}) = (O'',\Omega'')$ and let $X_m$ be their isomorphic sequence of elementary extensions of $X_{n_k}$ ending in $X_{m_k}=X''$. Thus, \ref{P: increasing wild animals} is satisfied.
    Let $D_{m-1}$ be the disk in $\Omega_{m-1}$ over which the connected sum is performed, so \ref{J: conn sum} and the first part of \ref{P: compatability of f_n} are satisfied.
    It is easy to construct the maps $\phi_m:\Omega_m \to \Omega_{m-1}$ such that \ref{J: conn sum 1} and the second part of \ref{P: compatability of f_n} are satisfied. 
    By Remark \ref{rem: choosing disks properly in mitoses}, we may assume that $D_m$ were chosen inductively to avoid the curves $\phi_{m,n}^{-1}(\partial D_n)$ for all $n\le m$, so \ref{J: D_n are disjoint} is satisfied as well.
    
    It follows by construction that $\compdiam(X_{m_k},\Xi)\le \frac{1}{k+1}$. Since $X_{n_{k+1}}$ will be an extension of $X_{m_k}$, \ref{P': compdiam in Xi} will be satisfied.
    
    By construction, in every component $E_{n_k}$ of $\Xi-X_{n_k}$ the PH-graph $X_{m_{k}}=X''$ contains a non-planar (resp. a non-orientable) PH-graph if such exists. If $C_{n_k}$ is the corresponding component in $(O_{n_k},\Omega_{n_k})$ under the isomorphism $f_{n_k}$, along the sequence $(O_m,\Omega_m)$ a connected sum with $\Ttwo$ (resp. $\RPtwo$) must be performed along a disk $D_m$ such that $\phi_{m,n_k}(D_m) \subseteq C_{n_k}$. Thus, \ref{J': denseness} is satisfied.
    
    We are left to show \ref{P': compdiam in Omega}.
    
    \paragraph{Step 2. Refining the tame animal.}
    Let $\mathcal{U}$ be an open cover of $\Omega_{m_k}$ such that $\diam_{d_s}(\phi_{m_k,s}(U))\le \frac{1}{k+1}$ for all $s\le n_{k}$ and all $U\in\mathcal{U}$. 
    By Lemma \ref{PONG}, there exists an extension $X'''$ of $X_{m_k}$ whose organism structure is isomorphic to an extension $(O''',\Omega''')$ that is obtained from $(O_{m_k},\Omega_{m_k})$ by a sequence of $t$ benign mitoses, such that each component of $\Omega'''-O'''$ is contained in some $U\in\mathcal{U}$.
    Set $n_{k+1} = m_k + t$, let $(O_m,\Omega_m)$ for $m_k \le m \le n_{k+1}$ be the sequence of $t$ benign mitoses starting with $(O_{m_k},\Omega_{m_k})$ and ending with $(O_{n_{k+1}},\Omega_{n_{k+1}}) = (O''',\Omega''')$. Similar to Step 1, let $X_m,\phi_m, D_m, f_m$ be chosen to satisfy \ref{J: conn sum},\ref{J: conn sum 1},\ref{J: D_n are disjoint},\ref{P: increasing wild animals},\ref{P: compatability of f_n}. 
    
    Clearly, \ref{P': compdiam in Xi},\ref{J': denseness} are still satisfied, and by construction \ref{P': compdiam in Omega} is now satisfied. This completes the proof. \qed
	}

\bibliographystyle{abbrv}
\bibliography{surface_boundaries_bib}

 \end{document}